\documentclass[12pt,a4paper,oneside]{amsart}

\usepackage{amsfonts, amsmath, amssymb, amsthm, amscd}
\usepackage{hyperref}
\hypersetup{
  colorlinks   = true, 
  urlcolor     = blue, 
  linkcolor    = blue, 
  citecolor   = red 
}
\usepackage{mathtools}

\usepackage[utf8]{inputenc}

\usepackage{anysize}
\marginsize{2.6cm}{2.6cm}{2.6cm}{2.6cm}
\usepackage[pdftex]{graphicx}

\newtheorem{theorem}{Theorem}
\newtheorem{lemma}[theorem]{Lemma}

\newtheorem{corollary}[theorem]{Corollary}

\theoremstyle{definition}

\theoremstyle{remark}
\newtheorem{remark}[theorem]{Remark}

\numberwithin{equation}{section}

\DeclareMathOperator{\conv}{conv}

\renewcommand{\epsilon}{\varepsilon}

\renewcommand{\phi}{\varphi}

\newcommand{\supp}{\mathop{\mathrm{supp}}\nolimits}

\newcommand{\su}{\mathcal{A}}

\newcommand{\R}{\mathbb{R}}

\newcommand{\Sp}{\mathbb{S}}

\title{Bisecting measures with hyperplane arrangements}

\author{Alfredo Hubard {$^\spadesuit$}}

\email{alfredo.hubard@u-pem.fr} 

\address{{$^\spadesuit$} Universit\'e Paris-Est Marne-la-Vall\'ee, LIGM. UMR 8049, CNRS, ENPC, ESIEE, UPEM, F-77454, Marne-la-Vall\'ee, France.}

\author{Roman Karasev {$^\clubsuit$}}

\email{r\_n\_karasev@mail.ru}
\urladdr{http://www.rkarasev.ru/en/}

\address{{$^\clubsuit$} Moscow Institute of Physics and Technology, Institutskiy per. 9, Dolgoprudny, Russia 141700}
\address{{$^\clubsuit$} Institute for Information Transmission Problems RAS, Bolshoy Karetny per. 19, Moscow, Russia 127994}

\thanks{{$^\clubsuit$} Supported by the Russian Foundation for Basic Research grant 18-01-00036}

\begin{document}

\begin{abstract}
We show that any $nD$ measures in $\R^n$ can be bisected by an arrangement of $D$ hyperplanes, when $n$ is a power of two.
\end{abstract}

\maketitle

\section{Introduction}

Let $\mathcal{H}=\{H_1, H_2, \ldots H_D\}$ be a finite set of hyperplanes,  $\{A_1,A_2\ldots A_D\}$ affine functions such that the zero set of $A_i$ is $H_i$, and $P^{\mathcal{H}}=A_1 A_2 \ldots A_D$ the product of these affine functions.
If $\mu$ is a measure in $\R^n$, we will say that $\mathcal{H}$ \emph{bisects} $\mu$ if 
\[
\mu\left\{v \in \R^n: P^{\mathcal{H}}(v)>0\right\}\leq \frac{\mu(\R^n)}{2}\quad\text{and}\quad \mu\left\{v \in \R^n: P^{\mathcal{H}}(v)<0\right\}\leq \frac{\mu(\R^n)}{2}.
\] 


\begin{theorem}
\label{multiham} 
Let $n$ and $D$ be integers such that $D>0$ and $n>1$ is a power of two. Given $nD$ finite measures $\mu_1,\mu_2, \ldots \mu_{nD}$ in $\R^{n}$, there exists an arrangement of at most $D$ hyperplanes that bisect each of the measures.
\end{theorem}

Observe that a family of $nD+1$ delta masses based at a set of points, no $n+1$ of which lie on the same hyperplane cannot be simultaneously bisected by less than $D+1$ hyperplanes. Barba and Schnider~\cite{Pizza} conjectured that the previous theorem holds for any $n$ and confirmed this conjecture for the case of four measures in the plane ($n=D=2$). 
Notice that the case $D=1$ of this conjecture corresponds to the classical ham sandwich theorem (see the book \cite{Matousek} for many other ham sandwhich type results).


\section{Parametrization of arrangements}

Parametrize hyperplanes in $\R^n$  by elements of $\Sp^n$ mapping $(a_0,a_1,\ldots, a_n) \in \Sp^n$ to the affine function
\[
A(x) = a_0 + a_1 x_1 + \dots + a_nx_n.
\]

Parametrize hyperplane arrangements by elements of $(\Sp^n)^D$.  An element of $(\Sp^n)^D$ corresponds to $D$ affine functions $A_1,\ldots, A_D$ and the polynomial corresponding to $\mathcal{H}=\{A_1^{1-}(0), A_2^{1-}(0), \ldots  A_D^{-1}(0)\}$  of degree $D$ is given by
\[
P^{\mathcal{H}}(x) = A_1(x) \dots A_D(x).
\]

Let $\mathfrak S_D$ be the symmetric group of permutations of $D$ elements, and $\mathbb{Z}/2^D$ be the $D$-fold product of the abelian group on two elements. Let $G=\mathfrak S_D \ltimes \mathbb{Z}/2^D$ be their semi-direct product.
 The group $\mathbb{Z}/2^D$ acts on $(\Sp^n)^D$ by the antipodal map $A\mapsto -A$ on each $\Sp^n$ factor, this action is free, the group $\mathfrak S_D$ acts on $(\Sp^n)^D$ by permuting the factors. Their semi-direct product acts by permuting and applying antipodal maps on some of the factors.

The action of $G$ on $(\Sp^n)^D$ is not free. Its non-free part $\Sigma$ corresponds to $D$-tuples $(A_1,\ldots, A_D)$ such that $A_i = A_j$ or $A_i = -A_j$ for some $i\neq j$.

\section{Approximation of measures}

We prove the theorem for a subspace of measures, which is a dense subset of $\mathcal{P}$, the space of Borel probability measures with the weak topology; then we deduce the general case by approximation. We denote by $\mathcal{P}^k$ its $k$-fold Cartesian product, whose elements are sets $\{\mu_1,\mu_2\ldots \mu_k\}$ of Borel probability measures in $\R^n$. The material that we need from measure theory is covered in many analysis books, see for instance \cite{tao2013,varadhanprobability}.

\begin{lemma}
\label{open-counterexample}
For $k,D>0$, the set of ordered $k$-tuples of Borel measures that are not bisectable by $D$ hyperplanes is open in $\mathcal{P}^k$.
\end{lemma}
\begin{proof}
Assume that $M=(\mu_1,\ldots,\mu_k) \in \mathcal{P}^{k}$ is a $k$-tuple of Borel probability measures that cannot be bisected by an arrangement of $D$ hyperplanes $\mathcal H$. For any polynomial $P^{\mathcal H}$, there exists a sign $\pm$ and an $i\in \{1\,\ldots,k\}$ such that
\[
\mu_i \left\{ \pm P^{\mathcal H} > 0 \right\} > 1/2.
\]
From continuity of the measure $\mu_i$ we can choose an open set $W$ whose closure is compact and is contained in $\left\{ \pm P^{\mathcal H} > 0 \right\}$ such that,
\[
\mu_i(W) > 1/2.
\]
By definition, the ordered $k$-tuples of Borel probability measures $M' = (\mu'_1,\ldots,\mu'_k)$ such that $\mu'_i(W)>1/2$ constitute a neighborhood $\mathcal U \ni M$ in the weak topology.
The arrangements of hyperplanes $\mathcal H'$ such that $P^{\mathcal H'}$ is positive on the closure of $W$ constitute a neighborhood $\mathcal V\ni \mathcal H$ in the topology on the space of arrangements. Any pair of $M'\in\mathcal U$ and $\mathcal H'\in\mathcal V$ have the property that $\mathcal H'$ does not bisect $M'$.

Since the space of arrangements $(\Sp^n)^D$ is compact, a finite number of such $\mathcal V_1,\ldots, \mathcal V_N$ cover the whole space of arrangements. The intersection of the respective $\mathcal U_1,\ldots,\mathcal U_N$ produce a neighborhood of $M$ every member of which cannot be bisected with any arrangement of hyperplanes. 
\end{proof}


\begin{corollary}\label{open}Theorem \ref{multiham} for Borel measures follows from its validity on any dense subset of $\mathcal{P}^k$.
\end{corollary}

Denote by  $\delta_v$ the Dirac delta mass at the point $v$, i.e. for a Borel set $X$, $\delta_{v}(X)=1$ if $v \in X$ and $\delta_v(X) = 0$ otherwise. We call measures of the form $\frac{1}{N}\sum_{k=1}^N \delta_{v_k}$ with odd $N$ and $v_1,\ldots, v_N$ in general position, \emph{oddly supported measures}. We say that a finite family of measures is in \emph{general position} if no hyperplane intersects $n+1$ connected components of the union of their supports. 

\begin{lemma}
\label{dense-discrete}
Oddly supported measures in general position are dense in $\mathcal{P}$. Ordered $k$-tuples of oddly supported measures in general position are dense in $\mathcal{P}^k$.
\end{lemma}

\begin{proof}
Assume the contrary, then there is a Borel probability measure $\mu$ whose weak neighborhood $\mathcal V$ contains no oddly supported measure. It is sufficient to consider $\mathcal V$ from the base of the weak topology given by a finite set of inequalities 
\[
\mu'(U_1) > m_1,\ldots, \mu'(U_\ell) > m_\ell
\]
for open $U_i$ and real $m_i$. Let $N$ be an odd number. Sample $N$ points $v_k$ independently, distributed according to $\mu$ and consider the random measure
\[
\nu_N = \frac{1}{N} \sum_{k=1}^N \delta_{v_k}.
\]
The random variable $\nu_N(U_i)$ is given by,
\[
\nu_N(U_i) = \frac{\#\{k=1,\ldots, N : v_k \in U_i\} }{N}
\]
This is a sum of $N$ independent Bernoulli random variables with expectation $\mu(U_i)$. By the law of large numbers $\nu_N(U_i)$  converges almost surely to $\mu(U_i)$. Hence for sufficiently large $N$ the probability of satisfying the inequalities $\nu_N(U_i) > m_i$ simultaneously is arbitrarily close to $1$; and we might perturb the points $v_k$ so that none of them leave any $U_i$ it belonged to so that for the perturbed measure $\nu_N(U_i)$ is still in $\mathcal V$. 

The second statement follows immediately, we do the same for $k$ measures and take a single sufficiently large odd $N$. After that we perturb the total $Nk$ support points so that none of them leaves any $U$ (from the definition of a weak neighborhood) it belonged to.
\end{proof}

Let $\eta_{v}$ be an $\epsilon$-smoothening of the delta mass at $v$. More precisely $\eta_v$ is a Borel probability measure centrally symmetric around $v$, which is supported inside a ball $B_v(\epsilon)$ of radius $\epsilon$ centered at $v$ and has a continuous density. Now take points in general position $v_1,\ldots, v_N$ and consider a sum of $\epsilon$-smoothenings
\[
\mu=\frac{1}{N} \sum_{k=1}^N \eta_{v_k}.
\] 
If $N$ is an odd number and no $n+1$ tuple of the $B_{v_k}(\epsilon)$ are intersected by a hyperplane, then we include $\mu$ in the set $\mathcal{M}_\epsilon$. Finally we put $\mathcal{M}:=\cup_{\epsilon>0} \mathcal{M}_\epsilon$, this is the set of measures we will work with.

\begin{lemma} 
\label{density} 
The set $\mathcal{M}$ is dense in the space of probability measures with the weak topology, moreover the set of ordered $k$-tuples of measures in general position in $\mathcal{M}^k$ is dense in $\mathcal{P}^k$  
\end{lemma}

\begin{proof}  
For any oddly supported measure, we weakly approximate every delta mass $\delta_{v_k}$ by its respective $\eta_{v_k}$ supported in the respective $B_{v_k}(\epsilon)$. If $\epsilon$ is sufficiently small then no $n+1$ of the balls will be intersected by a single hyperplane. So $\mathcal{M}$ is dense in the space of oddly supported measures which by Lemma \ref{dense-discrete}, is dense in $\mathcal{P}$. Similarly, $\mathcal{M}^k$ is dense in $\mathcal{P}^k$. 
\end{proof}

\section{Bisecting well separated sets of measures}

We say that a family of sets $X_1,X_2\ldots X_m$ in $\R^n$ is \emph{well separated} if no $n$-tuple of their convex hulls $\conv(X_1), \conv(X_2)\ldots \conv(X_m)$ is intersected by an $(n-2)$-dimensional affine space. A family of measures is \emph{well separated} if their supports are well separated. The following lemma was shown in \cite{bhj} for absolutely continuous measures.

\begin{lemma}
\label{unique} 
For any family of well separated measures in general position $\mu_1,\mu_2, \ldots \mu_n \in \mathcal{M}$ there exists a unique hyperplane $H$ that bisects each of the measures.
\end{lemma}

\begin{proof}
The existence of this hyperplane is provided by the ham sandwich theorem, we only need to show the uniqueness. 
Assume we have a pair of halving hyperplanes $H$ and $H'$, since the measures are well-separated, the intersection $H\cap H'$ does not touch the convex hull of the support of some $\mu_i$. The both hyperplanes must intersect the interior of the support of $\mu_i$, since it is constructed from an odd number of equal measures. Now it is clear that one of the halves $H_-\cap \conv\supp \mu_i$ and $H_+\cap \conv\supp \mu_i$ strictly contains some of $H'_-\cap \conv\supp \mu_i$ and $H'_+\cap \conv\supp \mu_i$ and therefore $H$ and $H'$ cannot equipartition $\mu_i$ at the same time.
\end{proof}

The following lemma describes the bisecting arrangements of hyperplanes in the case when the measures are well separated.

\begin{lemma}
\label{coro} 
For any family of $nD$ well separated measures in general position, an arrangement of $D$ hyperplanes $\mathcal{H}$  is bisecting, if and only if, there is bijection $\phi$ between the elements of $\mathcal{H}$ and a partition of $[nD]$ into $n$-tuples such that the hyperplane $H_i \in \mathcal{H}$ bisects the $n$-tuple of measures with indices in $\phi(H_i)$. \end{lemma}

\begin{proof}
Given a partition $Y$ of $[nD]$ into $n$-tuples  $\{Y_1,Y_2\ldots Y_D\}$. By Lemma \ref{unique}, for each $n$-tuple $Y_i$, the corresponding measures are bisected by a unique ham sandwich cut, this defines a bijection $\phi^{-1}\colon Y \to \mathcal{H}$. Since the measures are well separated, any measure with index not in $\phi(H)$ is not intersected by $H$. So the arrangement is simultaneously bisecting. Conversely, since the supports are well separated, each hyperplane of a bisecting arrangement must intersect the supports of precisely $n$ of the measures, otherwise at least one measure cannot be bisected. In this situation each hyperplane bisects $n$ measures and does not touch the convex hulls of the supports the remaining measures. By Lemma \ref{unique}, such a hyperplane must be the unique ham sandwich cut of the corresponding $n$-tuple of measures.
\end{proof}

Let $N(n,D)$ be the number of unordered partitions of a set of $nD$ elements into $D$ sets of $n$ elements each. Clearly
\[
N(n,D) = \frac{(nD)!}{D!(n!)^D},
\]
but we will not use this formula.

\begin{lemma}
\label{parity} 
If $n$ is a power of two then $N(n,D)$ is odd.
\end{lemma}

\begin{proof}
Consider the action of the $2$-Sylow subgroup $S\subset\mathfrak{S}_{nD}$ on these partitions. To describe this Sylow subgroup we need to make a binary tree with $2^m$ leaves, where $2^m$ is the smallest power of two not smaller than $nD$. Then we drop the leaves that have numbers strictly greater than $nD$ and drop the corresponding higher vertices of the tree. Then $S$ is the symmetry group of the remaining tree and its embedding into $\mathfrak{S}_{nD}$ is obtained by looking at the leaves of the tree and how they are permuted by $S$. It is a Sylow subgroup just because by construction its order equals
\[
2^{\sum_{k\ge 1} \lfloor nD/2^k \rfloor},
\]
which is the largest power of two that divides $(nD)! = |\mathfrak S_{nD}|$.

The set $nD$ has a decomposition into consecutive $n$-tuples $P_1\cup\dots\cup P_D$. As it is easily seen, when $n$ is a power of two, each $P_i$ corresponds to a full binary subtree. Hence the group $S$ can permute transitively each of $P_i$ while fixing all elements of the other $P_j$, $j\neq i$. This guarantees that an unordered partition into $n$-tuples that is fixed by $S$ must coincide with the chosen partition $P_1\cup\dots\cup P_D$. Other partitions are not fixed under the $S$ action, so they come in orbits. Since $S$ is a $2$-group, all such orbits are even, hence the total number of partitions into $n$-tuples is odd.
\end{proof}

\section{Proof of the Theorem}

By Lemmas \ref{density} and Corollary \ref{open} it is sufficient to prove the theorem for measures in $\mathcal{M}$ (smoothened oddly supported measures) in general position. Denote by $\su_i$ the support of the measure $\mu_i$ and by $C_i$ the set of centers of balls whose union is $\su_i$. We say that the points on $\su_i$ are of \emph{color} $i$. Denote by $M$ the family $\{\mu_1,\mu_2,\ldots, \mu_{nD}\}$.
Arguing similarly to Lemma~\ref{unique} observe that for any family of measures in general position in $\mathcal{M}$ (not necessarily well-separated) a bisecting arrangement has to be the union of $D$ hyperplanes each of which intersects a heterochromatic set of $n$ connected components, otherwise some of the measures will not be bisected. We only need to count such arrangements.

We deform the measures $\mu_i$ continuously to a situation where we can easily count the number of bisecting arrangements of the family. We use measures in $\mathcal{M}$ throughout, so we might prescribe a trajectory of $C_i$ and choose $\epsilon>0$ later. In the following all the objects that we deal with depend on $t \in [0,1]$ which we call time, and denote this time with a subindex $t$. For each $t\in [0,1]$ we consider a measure $\mu_{i,t} \in \mathcal{M}$ that depends continuously on $t$ such that $\mu_{i,0}:=\mu_i$ and the family $M_1:=\{\mu_{1,1},\mu_{2,1},\ldots \mu_{nD,1}\}$ (at time $t=1$) is well separated and in general position. By Lemma \ref{coro} we know that the the family $M_1$ has exactly $N(n,D)$ bisecting arrangements.  

Let us further describe the motion of $M_t$ in more detail. We want to describe a \emph{generic trajectory} of measures in general position. Consider a point $b_i$ from a general position set $b_1,b_2, \ldots b_{nD}$. Choose $\alpha>0$ so that the balls $B(b_i, \alpha)$ are well separated. Then move each of the points of $C_i$ towards $b_i$ in such a way that each set $C_I$ is always in general position within itself and at the end, the support of the $\mu_{i,1}$ is contained in $B(b_i, \alpha)$. For example, the deformation could follow a homothety with center $b_i$. By perturbing the speed of the trajectories if necessary, we can assume that at no moment of time there exist two $(n+1)$-tuples of connected components of the $\su_i$, each of which is intersected by a hyperplane. In particular, at no time $t$, an $(n+2)$-tuple of connected components is intersected by a single hyperplane.  To put it short, in a generic trajectory the events when some $n+1$ supporting balls of the measures can be intersected by a hyperplane come one by one.

Denote by $Z_t$ the subset of points of $(\Sp^n)^D$ corresponding to bisecting arrangements of the family $M_t$. Our \emph{crucial observation} is that $Z_t$ does not touch the non-free part $\Sigma\subset (\Sp^n)^D$. An assumed $G$-fixed point of $Z_t$ corresponds to a set of hyperplanes in which two of the hyperplanes coincide. From the assumption on the generic trajectory it follows that we thus have at most $D-1$ distinct hyperplanes that intersect at least $nD$ supporting balls of the measures in the set $M_t$. But there is a unique $(n+1)$-tuple of such balls that can be intersected by a single hyperplane, in all other situations the hyperplanes intersect at most $n$ balls each. The inequality $n(D-1)+1 < nD$ thus gives a contradiction, so the non-free part of the space of arrangements is not touched during the motion.

Let us show that the parity of the number of bisecting arrangements stays invariant during the motion; then Lemma \ref{parity} delivers the result in the case we are interested in.

Consider the continuous $G$-equivariant map $f \colon (\Sp^n)^D \times [0,1] \to \R^{nD}$ given by 
\[
(f_t(x))_{i}=\mu_{i,t}\{P>0\}-\mu_{i,t}\{P<0\},
\] 
where $P$ is the polynomial we associate to $x \in (\Sp^n)^D$. We have the solution set $Z_t=f_t^{-1}(0) \subset (\Sp^n)^D\setminus \Sigma$ at time $t$. We need to show that $Z_0\neq \emptyset$ and let us assume the contrary, that $Z_0 = \emptyset$. 

If fact, the union of all such $Z_t$ for $t\in [0,1]$ is the preimage of zero $Z = f^{-1}(0)$, a closed subset of the product $(\Sp^n)^D \times [0,1]$, not touching the non-free part of this product $\Sigma\times [0,1]$. Denote the free part 
\[
F = \left( (\Sp^n)^D\setminus \Sigma\right) \times [0,1]
\]
for brevity. Using the Thom transversality theorem \cite{Thom1954, milnor1997} (on the free part $F$ we just apply the non-equivariant transversality for the sections of the vector bundle $F\times_G \mathbb R^{nD}$ over $F/G$) we modify $f$ in a neighborhood of $Z$ (not touching $\Sigma$) to produce a smooth $G$-equivariant map $f'$ transversal to zero, thus having $Z' = f'^{-1}(0)$ a submanifold with boundary in $F$. We can also choose $f'$ coniciding with $f$ for $t=0,1$ since for $t=0,1$ the map $f$ was already transversal to zero. Now we have $Z'$ with $Z'_0 = Z_0 = \emptyset$ and such that $Z'_1$ consists of an odd number of $G$-orbits. But $Z'_1$ is the boundary of the one-dimensional compact manifold $Z'$ with free action of $G$ that cannot consist of an odd number of $G$-orbits, a contradiction.

\begin{remark} 
The previous version of this paper incorrectly claimed Theorem \ref{multiham} for any $n$. It was claimed that the cohomology class
that was denoted there by $e_i$ vanished on the complement of the set of arrangements of $D$ hyperplanes bisecting a single measure. Actually the argument given there with the curve $\gamma_i$ provides this fact for the class $\sum_{i=1}^D e_i$, the modulo two Euler class of the one-dimensional representation of $(\mathbb Z/2)^D$, on which each generator of every $\mathbb Z/2$ acts antipodally. 
The vanishing lemma implies that if $(e_1 + \dots + e_D)^k$ is nonzero in the cohomology ring of the product of projective spaces, then for every $k$ measures there exist an arrangement of $D$ hyperplanes bisecting the measures. 
This in turn amounts to finding an odd multinomial coefficient,
$\binom{k}{k_1\ k_2\ \dots\ k_D} = \binom{k}{k_1}\cdot\binom{k-k_1}{k_2}\dots\binom{k - k_1 - \dots - k_{D-1}}{k_D}$
with $k_1,\ldots,k_d \le n$. For such a coefficient to be odd, when we add the numbers in the sum $k_1 + \dots + k_D$ in binary representation then no carry should occur. Consider the largest $m$ such that, $2^m \le n$, then, we need 
$
k_1 + \dots + k_D \le 2^{m+1} - 1.
$
There is an example of such a sum with no carry if we put for $D\le m+2$,
$
k_1 = 2^m, k_2 = 2^{m-1}, \ldots, k_{D-1} = 2^{m-D+2}, k_D =  2^{m-D+2} - 1,
$
and for $D\ge m+2$,
$
k_1 = 2^m, k_2 = 2^{m-1}, \ldots, k_m = 2, k_{m+1} = 1, k_{m+2} = \dots = k_D = 0.
$
From which we can conclude that we can bisect $2^{m+1}-1\le 2n-1$ measures with at most $2$ hyperplanes, and taking more hyperplanes, does not yield anything new with this technique (not using the permutations $\mathfrak S_D$).

On the other hand, if $2^m\le n$ and we have $2^m D$ measures in $\R^n$, we can project linearly to $\mathbb R^{2^m}$, apply Theorem \ref{multiham} to obtain a bisecting arrangement of $D$  hyperplanes in $\R^{2^m}$ and look at their inverse image, an arrangement of $D$ hyperplanes in $\R^n$ that bisects the original measures. Since $2^m D > 2^{m+1} - 1$ in the nontrivial case $D\ge 2$, Theorem \ref{multiham} always provides a better result then the above cohomological argument.
\end{remark}

\subsubsection*{Acknowledgements.}
We acknowledge and thank Pavle V. M. Blagojevi\'{c} for corrections and useful comments.

\bibliographystyle{plain}
\bibliography{main}

\begin{thebibliography}{1}

\bibitem{bhj}
Imre B{\'a}r{\'a}ny, Alfredo Hubard, and Jes{\'u}s Jer{\'o}nimo.
\newblock Slicing convex sets and measures by a hyperplane.
\newblock {\em Discrete {\&} Computational Geometry}, 39(1):67--75, Mar 2008.

\bibitem{Pizza}
Luis Barba and Patrick Schnider.
\newblock Sharing a pizza: bisecting masses with two cuts.
\newblock In {\em Proceedings of the 29th Canadian Conference on Computational
  Geometry}, CCCG'17, July 2017.

\bibitem{Matousek}
Jiri Matousek.
\newblock {\em Using the Borsuk-Ulam Theorem: Lectures on Topological Methods
  in Combinatorics and Geometry}.
\newblock Springer Publishing Company, Incorporated, 2007.

\bibitem{milnor1997}
J.W. Milnor.
\newblock {\em Topology from the Differentiable Viewpoint}.
\newblock Princeton Landmarks in Mathematics. Princeton University Press, 1997.

\bibitem{tao2013}
T.~Tao.
\newblock {\em An Introduction to Measure Theory}.
\newblock Graduate studies in mathematics. American Mathematical Society, 2013.

\bibitem{Thom1954}
R\'en\'e Thom.
\newblock Qu\'elqu\'es propri\'et\'es globales des vari\'et\'es
  diff\'erentiables.
\newblock {\em Commentarii Mathematici Helvetici}, 28:17--86, 1954.

\bibitem{varadhanprobability}
S.R.S. Varadhan.
\newblock {\em Probability Theory}.
\newblock American Mathematical Society, 2001.

\end{thebibliography}
\vfill

%
%


\end{document}